\documentclass{article}
\usepackage[left=30truemm, right=30truemm]{geometry}
\usepackage{amsmath,amssymb,amsthm}
\usepackage{bm}
\usepackage{mathtools}
\usepackage{xcolor}

\theoremstyle{definition}
\newtheorem{theorem}{Theorem}[]
\newtheorem{proposition}[theorem]{Proposition}

\newtheorem{lemma}[theorem]{Lemma}

\newtheorem{conjecture}[theorem]{Conjecture}

\numberwithin{equation}{section}

\title{Odd 4-coloring of outerplanar graphs}
\author{
  Masaki Kashima\thanks{School of Fundamental Science and Technology,
  Graduate School of Science and Technology, Keio University, Yokohama, Japan. email: masaki.kashima10@gmail.com}, 
  Xuding Zhu\thanks{School of Mathematical Sciences, Zhejiang Normal University, China. email: xdzhu@zjnu.edu.cn}
}

\begin{document}
\maketitle

\begin{abstract}
A proper $k$-coloring of $G$ is called an odd coloring of $G$ if for every vertex $v$, there is a color that appears at an odd number of neighbors of $v$.
This concept was introduced recently by Petru\v{s}evski and \v{S}krekovski, and they conjectured that every planar graph is odd 5-colorable.
Towards this conjecture, Caro, Petru\v{s}evski, and \v{S}krekovski showed that every outerplanar graph is odd 5-colorable, and this bound is tight since the cycle of length 5 is not odd 4-colorable.
Recently, the first author and others showed that every maximal outerplanar graph is odd 4-colorable.
In this paper, we show that a connected outerplanar graph $G$ is odd 4-colorable if and only if $G$ contains a block which is not a copy of the cycle of length 5.
This strengthens the result by Caro, Petru\v{s}evski, and \v{S}krekovski, and gives a complete characterization of odd 4-colorable outerplanar graphs.
\end{abstract}

\section{Introduction}\label{sec:intro}

Throughout this paper, we only consider simple, finite, and undirected graphs.
For a positive integer $k$, let $[k]$ denote the set of positive integers at most $k$.
For a graph $G$, a $k$-coloring of $G$ is a map $\varphi$ from the vertex set $V(G)$ to the set $[k]$ such that $\varphi(u)\neq\varphi(v)$ for every edge $uv$ of $G$.
For a proper $k$-coloring $\varphi$ of $G$, a vertex $v$ of $G$ {\em satisfies the odd condition (respectively, even condition) with respect to $\varphi$} if 
$|\varphi^{-1}(i) \cap N_G(v)|$ is odd for some color $i\in [k]$ (respectively, $|\varphi^{-1}(i) \cap N_G(v)|$ is even for some color $i \in [k] \setminus \{\varphi(v)\}$). 
A $k$-coloring of a graph $G$ is called an {\em odd $k$-coloring} of $G$ if every non-isolated vertex $v$ of $G$ satisfies the odd condition with respect to $\varphi$. 
For a graph $G$, the {\em odd chromatic number}, denoted by $\chi_o(G)$, is the least integer $k$ such that $G$ admits an odd $k$-coloring.
This concept was introduced by Petru\v{s}evski and \v{S}krekovski~\cite{Petrusevski}, and has been actively studied in the literature \cite{Anderson,Caro,Cho,Liu,Petr,Wang}.
One major problem is the odd chromatic number of planar graphs. 
The following conjecture was posed by Petru\v{s}evski and \v{S}krekovski~\cite{Petrusevski}.

\begin{conjecture}[\cite{Petrusevski}]\label{conj:planar}
  Every planar graph is odd 5-colorable.
\end{conjecture}

If Conjecture \ref{conj:planar} is true, then the bound is best possible since the cycle of length 5 is not odd 4-colorable.
Towards this conjecture, Petru\v{s}evski and \v{S}krekovski~\cite{Petrusevski} showed that every planar graph is odd 9-colorable, and Petr and Portier~\cite{Petr} improved the bound to 8.
For an integer $k\geq 4$, let $\mathcal{P}_k$ denote the family of planar graphs with girth at least $k$.
Cho et al.~\cite{Cho} showed that $\chi_o(G)\leq 6$ for every graph $G\in \mathcal{P}_5$, and that $\chi_o(G)\leq 4$ for every graph $G\in\mathcal{P}_{11}$.
The latter result is improved recently by Anderson et al.~\cite{Anderson}, where they showed that $\chi_o(G)\leq 4$ for every $G\in \mathcal{P}_{10}$.

In this paper, we focus on outerplanar graphs.
An outerplanar graph is a graph that can be embedded in the plane so that all vertices lie in the boundary of the outer face. An outerplanar graph $G$ is called a maximal outerplanar graph if for any nonadjacent vertices $u$ and $v$ of $G$, $G+uv$ is not an outerplanar graph.
Caro et al.~\cite{Caro} showed the following theorem.

\begin{theorem}[\cite{Caro}]\label{thm:outerplanar}
    Every outerplanar graph is odd 5-colorable.
\end{theorem}

As $C_5$ is an outerplanar graph,   the bound 5 is best possible, where $C_5$ is the cycle of length 5.
On the other hand,  the following theorem was proved in \cite{Kashima}.

\begin{theorem}[\cite{Kashima}]\label{thm:maximalouterplanar}
    For every maximal outerplanar graph $G$ and every list assignment $L: V(G)\to 2^{\mathbb{N}}$, if $|L(v)|\geq 4$ for any vertex $v$ of $G$, then $G$ admits an odd coloring $\varphi$ such that $\varphi(v)\in L(v)$ for every vertex $v$ of $G$.
    In particular, every maximal outerplanar graph is odd 4-colorable.
\end{theorem}

Note that deleting edges can increase the odd chromatic number of a graph. One natural question is which outerplanar graphs are odd 4-colorable.
This paper answers this question, and characterizes the family of odd 4-colorable outerplanar graphs.

\begin{theorem}\label{thm:main1}
    A connected outerplanar graph $G$ is odd 4-colorable if and only if $G$ has a block which is not a copy of $C_5$.
\end{theorem}

For the purpose of using induction,  we prove a slightly stronger statement.
Assume $G$ is a graph and $v$ is a vertex of $G$. We denote by $(G,v)$ the {\em rooted graph} with $v$ be the {\rm root vertex} of $G$. 
An odd $k$-coloring of $(G,v)$ is an odd $k$-coloring $\varphi$ of $G$ such that vertex $v$ satisfies both the odd condition and the even condition.
Note that if $d_G(v)=2$, then any odd 4-coloring of $G$ is an odd 4-coloring of $(G,v)$, as there is a color $i\in [4]\setminus \{\varphi(v)\}$ such that  $|\varphi^{-1}(i)\cap N_G(v)|=0$.

\begin{theorem}\label{thm:main2}
    A connected rooted outerplanar graph $(G,v)$ admits an odd 4-coloring if and only if $G$ has a block which is not a copy of $C_5$.
\end{theorem}

It is obvious that Theorem \ref{thm:main2} implies Theorem \ref{thm:main1}.
It was proved in \cite{Caro2} if each block of a graph $G$ is isomorphic to $C_5$, then $G$ is  not proper conflict-free 4-colorable. The same argument shows that $G$ is not odd 4-colorable. On the other hand, such a graph is ``almost" odd 4-colorable in the sense that only one vertex does not satisfy the odd condition. 

\begin{proposition}\label{prop:c5family} 
    If $G$ is a connected graph and every block of $G$ is isomorphic to $C_5$, then $G$ is not odd 4-colorable. On the other hand, for each vertex $v$ of $G$, there is a proper $4$-coloring $\varphi$ of $G$ such that any vertex $u \ne v$ satisfies the odd condition and $v$ satisfies the even condition with respect to $\varphi$.
\end{proposition}

\begin{proof}
    The proof goes by induction on the number of blocks of $G$.
    If $G$ has only one block, then $G =(v_1v_2v_3v_4v_5)$ is a  cycle of length 5, and thus $G$ is not odd 4-colorable. On the other hand, suppose that $v=v_4$ by symmetry, and let $\varphi(v_j)=j$ for $j\in\{1,2,3,4\}$, and $\varphi(v_5)=3$.
    Then for each $j \ne 4$, $|\varphi^{-1}(i) \cap N_G(v_j)|$ is odd for some color $i$. 
    The vertex $v_4$ satisfies the even condition since $|\varphi^{-1}(1)\cap N_G(v_4)|=0$.

    Suppose that $G$ has at least 2 blocks, and $v$ is a vertex of $G$. 
    Let $B=v_1v_2v_3v_4v_5$ be a leaf block of $G$ containing a cut vertex $v_1$ of $G$, and $v \ne v_i$ for $i\in\{2,3,4,5\}$. Let $G'=G-\{v_2,v_3,v_4,v_5\}$. 
  
    If $\varphi$ is an odd $4$-coloring of $G$, then for each $i \in \{2,3,4,5\}$, the two neighbors of $v_i$ are colored by distinct colors.  This implies that $\varphi(v_2)=\varphi(v_5)$. Hence the restriction of $\varphi$ to $G'$ is an odd 4-coloring of $G'$, a contradiction. 

    On the other hand, there is a proper 4-coloring $\varphi$ of $G'$ such that any vertex $u \neq v$ satisfies the odd condition and $v$ satisfies the even condition. Assume $\varphi(v_1)=1$. Extend $\varphi$ to $G$ by letting $\varphi(v_i)=i$ for $i\in\{2,3,4\}$ and $\varphi(v_5)=2$, it is easy to check that any vertex $u \neq v$ satisfy the odd condition and the vertex $v$ still satisfies the even condition.
\end{proof}

\section{Color exchanging lemma}\label{sec:lemma}

\begin{lemma}\label{lem:colorexchange}
    Let $G$ be an outerplane graph, and let $v$ be a vertex of $G$.
    Let $xy$ be an edge on the boundary of the outer face which is not a cut edge of $G$.
    If $G$ admits an odd 4-coloring such that $v$ satisfies the even condition, then there exists an odd 4-coloring $\varphi$ of $G$ such that the vertex $v$ and at least one of $\{x, y\}$ satisfy the even condition with respect to $\varphi$.
\end{lemma}

\begin{proof}
    Suppose that $\varphi_0$ is an odd 4-coloring of $G$ such that $v$ satisfies the even condition.
    If $v\in \{x, y\}$, then there is nothing to prove, so assume that $v\notin \{x, y\}$.
    Without loss of generality, we may assume that $\varphi_0(x)=1$ and $\varphi_0(y)=2$.
    If either $x$ or $y$ satisfies the even condition with respect to $\varphi_0$, then we are done.
    We assume that none of $x$ and $y$ satisfies the even condition with respect to $\varphi_0$.
    Let $z$ be a cut vertex of $G-xy$ which separates $x$ and $y$.
    Let $G_x$ be a subgraph induced by the union of $\{z\}$ and the vertices of the component of $(G-xy)-z$ containing $x$, and let $G_y=G-(V(G_x)\setminus\{z\})$.
    Depending on the color of $z$, we consider the following two cases.

    \vspace{\baselineskip}
    \noindent
    \textbf{Case 1.} $\varphi_0(z)\in\{1,2\}$.

    Without loss of generality, we may assume that $\varphi_0(z)=1$.
    Let $\varphi_1$ be obtained from $\varphi_0$ by exchanging colors 2 and 3 in $G_x$; and 
    $\varphi_2$ be obtained from $\varphi_0$ by exchanging colors 2 and 4 in $G_x$.
    Then both $\varphi_1$ and $\varphi_2$ are proper colorings of $G$, and 
    the odd condition and the even condition at every vertex other than $x$ and $z$ are preserved in both colorings.
    We now consider the case $z=v$. (The case $z\neq v$ is easier.)
    When $d_G(z)$ is even, then $z$ satisfies the even condition with respect to both $\varphi_1$ and $\varphi_2$. 
    Indeed, if $z$ does not satisfy the even condition with respect to $\varphi_i$, then $|\varphi_i^-(2)\cap N_G(z)|$, $|\varphi_i^-(3)\cap N_G(z)|$ and $|\varphi_i^-(4)\cap N_G(z)|$ are all odd and thus $d_G(z)$ must be odd.
    Suppose $z$ does not satisfy the odd condition with respect to $\varphi_1$.
    Then each of colors 2 and 3 appears at an odd number of neighbors of $z$ with respect to $\varphi_0$, and both of them turn into even after the exchange of colors 2 and 3 in $G_x$.
    Hence color 3 appears at an odd number of neighbours of $z$ with respect to $\varphi_2$, and thus $z$ satisfies the odd condition with respect to $\varphi_2$.
    Similarly, when $d_G(z)$ is odd, $z$ satisfies the odd condition with respect to both $\varphi_1$ and $\varphi_2$, and satisfies the even condition with respect to at least one of $\varphi_1$ and $\varphi_2$.
    In any case, $z$ satisfies both the odd condition and the even condition with respect to at least one of $\varphi_1$ and $\varphi_2$. 
    Without loss of generality, we may assume that $z$ satisfies the odd condition and the even condition with respect to $\varphi_1$.

    Since $x$ does not satisfy the even condition with respect to $\varphi_0$, we conclude that $|\varphi_1^{-1}(3) \cap N_G(x)|=  |\varphi_0^{-1}(2) \cap N_G(x)|-1$ is even,  and  $|\varphi_1^{-1}(4) \cap N_G(x)|=  |\varphi_0^{-1}(4) \cap N_G(x)|$ is odd.
    Hence $x$ satisfies both the odd condition and the even condition with respect to $\varphi_1$. So $\varphi_1$ is a desired odd 4-coloring of $G$.

    \vspace{\baselineskip}
    \noindent
    \textbf{Case 2.} $\varphi_0(z)\notin\{1,2\}$.

    Assume $\varphi_0(z)=3$.
    Let $\varphi_3$ be obtained from $\varphi_0$ by exchanging colors 2 and 4 in $G_x$; and $\varphi_4$ be obtained from $\varphi_0$ by exchanging colors 1 and 4 in $G_y$.
    Similarly to Case 1, both $\varphi_3$ and $\varphi_4$ are proper colorings of $G$, and  at least one of $\varphi_3$ and $\varphi_4$  is an odd 4-coloring of $G$ such that the vertex $v$ satisfies the even condition.

    Since $x$ does not satisfy the even condition with respect to  $\varphi_0$, $|\varphi_3^{-1}(2) \cap N_G(x)| = |\varphi_0^{-1}(4) \cap N_G(x)|+1$ is even, and $|\varphi_3^{-1}(3)\cap N_G(x)|=|\varphi_0^{-1}(3)\cap N_G(x)|$ is odd.  
    Hence $x$ satisfies both the odd condition and the even condition with respect to $\varphi_3$.
    Similarly, the vertex $y$ satisfies both the odd condition and the even condition with respect to $\varphi_4$.
    Therefore one of $\varphi_3$ and $\varphi_4$ is a desired odd 4-coloring of $G$.
    This completes the proof of Lemma \ref{lem:colorexchange}.
\end{proof}

\section{An unavoidable set}\label{sec:unavoidable}

In this section, we define an unavoidable set of 2-connected outerplanar graphs which is needed for our proof of Theorem \ref{thm:main2}.
Let $G$ be a 2-connected outerplanar graph, and $v$ be a vertex of $G$.

\begin{itemize}
  \item An {\em ear} $H$ of $G$ is a cycle
   $(u_1, u_2, \dots , u_r)$ such that $d_G(u_i) =2$ for $i\in\{2,3,\dots, r-1\}$.  
  The edge $u_1u_r$ is  the {\em root edge} of $H$. 
  We say $H$ is \textit{good for $v$} if there exists $i \in \{1,r\}$ such that  $d_G(u_i)=3$ and $v \notin V(H)\setminus\{u_{r+1-i}\}$. 
  \item An {\em ear chain} $H$ is a sequence of ears $H_1,H_2,\dots, H_{s-1}$  ($s \geq 3$) such that the root edges of the ears form a cycle $(v_1, v_2, \dots , v_s)$ (the root edge of $H_i$ is $v_iv_{i+1}$) and  $d_G(v_i) = 4$ for $i\in\{2,3,\dots, s-1\}$. The edge $v_1v_s$ is the {\em root edge} of $H$.
  We say  $H$ is {\em good for $v$} if there exists $i \in \{1,s\}$ such that  $d_G(v_i) \le 5$ and $v \notin V(H)\setminus\{v_{s+1-i}\}$.
  
  \item An {\em ear double chain} $H$  consists of  a sequence of ear chains $H_1,H_2,\dots, H_{t-1}$ ($t \geq 3$)  whose root edges form a cycle $(w_1,w_2, \dots, w_t)$ (the root edge of $H_i$ is $w_iw_{i+1}$) and $d_G(w_i)=6$ for $i\in\{2,3,\dots, t-1\}$. The edge $w_1w_t$ is the {\em root edge} of $H$.    We say $H$ is {\em good for $v$} if  $v \notin V(H)\setminus\{w_1, w_t\}$.
  
\end{itemize}

\begin{lemma}\label{lem:unavoidable}
    Let $G$ be a 2-connected outerplanar graph with at least 4 vertices.
    Let $v$ be a vertex of $G$.
    If $G$ is not a cycle, then $G$ contains an ear,
or an ear chain,
or an ear double chain that is  good for $v$.
\end{lemma}

\begin{proof}
    Let $F$ be the outercycle of $G$.
    We consider the following three cases.

    \vspace{\baselineskip}
    \noindent
    \textbf{Case 1.} Every chord of $F$ is the root edge of some ear.

    Let $T$ be the subgraph of $G$ induced by chords of $F$.
    By our assumption, $T$ is a cycle or a union of paths.
    Suppose first that $T$ is a cycle $(v_1, v_2, \dots , v_s)$, and let $H_i$ be an ear whose root edge is $v_iv_{i+1}$ for each $i\in [s]$ ($v_{s+1}=v_1$).
    Without loss of generality, we may assume that $v\in V(H_{s})\setminus\{v_1\}$.
    Then, the sequence of ears $H_1, H_2, \dots , H_{s-1}$ forms an ear chain with the root edge $v_1v_s$.
    Since $d_G(v_1)=4$ and $v_1\neq v$, the ear chain is good for $v$.
    Hence we may assume that $T$ is a union of paths.
    Since there are at least two leaves of $T$, there is an edge $v_1v_2$ of $T$ such that $d_T(v_2)=1$ and the corresponding ear $H'$ of $G$ satisfies $v\notin V(H')\setminus\{v_1\}$, and $d_G(v_2)=3$. So $H'$ is a good ear for $v$.

    \vspace{\baselineskip}
    \noindent
    \textbf{Case 2.} Every chord of $F$ is the root edge of some ear or some ear chain.

Note that if a chord $xy$ of $F$  is the root of an ear $H$, and also the root of an ear chain $H'$, then $G = H \cup H'$, and every chord of $F$ is the root of  an ear. This falls into Case 1. 
Thus we assume that 
each chord of $F$ is either the root of an ear or the root of an ear chain, but not both. 
    By Case 1, we may assume that there is a chord of $F$ which is the root of an ear chain.

    Let $T$ be the subgraph of $G$ induced by chords that are the root edges of ears, and 
    let $T'$ be the subgraph of $G$ induced by chords that are the root edges of ear chains.
    By our assumption, $T'$ is a cycle or a union of paths.
    Suppose first that $T'$ is a cycle $(w_1, w_2, \dots , w_t)$, and let $H_i$ be an ear chain whose root edge is $w_iw_{i+1}$ for each $i\in [t]$ ($w_{t+1}=w_1$).
    Without loss of generality, we may assume that $v\in V(H_k)\setminus\{w_1\}$.
    Then, the sequence of ear chain $H_1, H_2, \dots , H_{k-1}$ forms an ear double chain with the root edge $w_1w_t$, that is good for $v$. 
    
    Assume that $T'$ is a union of paths.
    There is an edge $w_1w_2$ of $T'$ such that $d_{T'}(w_2)=1$ and  the corresponding ear chain $H'$ satisfies $v\notin V(H')\setminus\{w_1\}$.
    As $d_F(w_2)=2$ and $d_T(w_2) \le 2$, we have $d_G(w_2)=d_F(w_2)+d_T(w_2)+d_{T'}(w_2)\leq 5$.
    Hence $H'$ is a good ear chain for $v$.

    \vspace{\baselineskip}
    \noindent
    \textbf{Case 3.} There is a chord of $F$ which is neither the root edge of an ear nor the root edge of an ear chain.

    For each chord $xy$ of $F$ which is neither the root edge of an ear nor the root edge of an ear chain, let $F'$ be a cycle in $F+xy$ such that $v\notin V(F')\setminus\{x, y\}$.
    We choose such chord $xy$ for which the length of $F'$ is minimum, 
    and let $H$ be the subgraph of $G$ induced by the vertices of $F'$.
    Note that $H$ is a 2-connected outerplanar graph with the outercycle $F'$.
    By the minimality of $F'$, for each edge $e\in E(H)\setminus E(F')$, $e$ is either the root edge of some ear of $G$ or the root edge of some ear chain of $G$, which is contained in $H$. 
    
    Let $T_H$ be the subgraph of $H$ induced by edges which is the root edge of some ear of $H$, and let $T'_H$ be the subgraph of $H$ induced by edges which is the root edge of some ear chain of $H$.  
    Note that an ear or an ear chain $H'$ of $H$ is an ear or an ear chain of $G$, unless $xy$ is an edge of $H'$. 
    By definitions, each of $T'_H$ and $T_H$ is a cycle or a union of paths.
    
Assume $E(T'_H) \ne \emptyset$. If $T'_H$ is a cycle, then there is an ear chain $H'$ of $H$ that contains $xy$. Let $x'y'$ be the root edge of $H'$. Then $H-(V(H')\setminus\{x',y'\})$ is an ear double chain of $G$ that is good for $v$. If $T'_H$ is an $xy$-path, then $H$ is an ear double chain good for $v$ with the root edge $xy$.  Assume $T'_H$ is neither a cycle nor an $xy$-path. Then there is an edge $w_1w_2$ of $T'_H$ such that $d_{T'_H}(w_1) =1, w_1 \notin \{x,y\}$, and $w_1w_2$ is the root edge of an ear chain $H'$ which does not contain $xy$. As $d_{T_H}(w_1) \le d_F(w_1) =2$, we know that $d_G(w_1) \le 5$. Hence $H'$ is good for $v$.

Assume $E(T'_H) = \emptyset$. 
 If $T_H$ is a cycle, then there is an ear $H'$ of $H$ which contains the edge $xy$.  Let   $x'y'$ be the root edge of $H'$. 
 Then $H - (V(H')\setminus\{x',y'\})$  is an ear chain $H''$ of $H$ with root edge $x'y'$, contrary to our assumption. 
Assume $T_H$ is not a cycle.
Since $xy$ is not the root edge of an ear chain of $G$,   $T_H$ is not an $xy$-path of $H$. Hence there is an edge $v_1v_2$ of $T_H$ such that $d_{T_H}(v_1)=1$, $v_1\notin \{x, y\}$, and $v_1v_2$ is the root edge of an ear $H'$ of $H$ which does not contain $xy$.
Then $H'$ is an ear of $G$ that is good for $v$.   
This completes the proof of Lemma \ref{lem:unavoidable}.
\end{proof}

\section{Proof of Theorem \ref{thm:main2}}\label{sec:proof}

Assume Theorem \ref{thm:main2} is not true and $(G, v)$ is a counterexample with minimum number of vertices. It is obvious that $G$ is connected and has at least 5 vertices.
Let $\mathcal{G}_{C_5}$ be the family of graphs every block of which is isomorphic to the cycle of length 5. 
By Proposition \ref{prop:c5family}, $G \not\in \mathcal{G}_{C_5}$.

First we consider the case that $G$ is not 2-connected.

Assume $G$ is not 2-connected and  $x$ is a cut vertex of $G$. Let $G_1$ and $G_2$ be connected subgraphs of $G$ such that $V(G_1)\cap V(G_2)=\{x\}$, $E(G_1)\cap E(G_2)=\emptyset$, and $E(G_1)\cup E(G_2)=E(G)$.
Since $G\notin \mathcal{G}_{C_5}$, either $G_1$ or $G_2$ does not belong to $\mathcal{G}_{C_5}$.
Without loss of generality, we may assume that $G_1\notin \mathcal{G}_{C_5}$.

We first suppose that $v=x$.
By induction hypothesis, there is an odd 4-coloring $\varphi_1$ of $G_1$ such that $v$ satisfies the even condition.
Without loss of generality, we may assume that $\varphi_1(x)=1$, $|\varphi_1^{-1}(2)\cap N_{G_1}(x)|$ and $|\varphi_1^{-1}(3)\cap N_{G_1}(x)|$ have the different parities.
By induction hypothesis and Proposition \ref{prop:c5family}, there is a proper 4-coloring $\varphi_2$ of $G_2$ such that every vertex $u\neq x$ satisfies the odd condition.
Without loss of generality, we may assume that $\varphi_2(x)=1$ and 
$|\varphi_1^{-1}(2)\cap N_{G_1}(x)|$ and $|\varphi_1^{-1}(3)\cap N_{G_1}(x)|$ have the same parity where the latter statement follows from Pigeon-Hole Principle.
We define a coloring $\varphi$ of $G$ by $\varphi(u)=\varphi_i(u)$ for $i\in\{1,2\}$. 
Since one of $|\varphi^{-1}(2)\cap N_G(x)|$ and $|\varphi^{-1}(3)\cap N_G(x)|$ is odd and the other is even, $\varphi$ is a desired odd 4-coloring of $G$, a contradiction.

Suppose that $v\in V(G_1)\setminus\{x\}$.
By induction hypothesis, there is an odd 4-coloring $\varphi_1$ of $G_1$ such that $v$ satisfies the even condition.
Without loss of generality, we may assume that $\varphi_1(x)=1$ and $|\varphi_1^{-1}(2)\cap N_{G_1}(x)|$ is odd.
By induction hypothesis and Proposition \ref{prop:c5family}, there is a proper 4-coloring $\varphi_2$ of $G_2$ such that every vertex $u\neq x$ satisfies the odd condition and the vertex $x$ satisfies the even condition with respect to $\varphi_2$.
Without loss of generality, we may assume that $\varphi_2(x)=1$ and $|\varphi_2^{-1}(2)\cap N_{G_2}(x)|$ is even.
We define a coloring $\varphi$ of $G$ by $\varphi(u)=\varphi_i(u)$ for $i\in\{1,2\}$, and it is easy to check that $\varphi$ is a desired odd 4-coloring of $G$, a contradiction.

Suppose that $v\in V(G_2)\setminus\{x\}$.
If $G_2\notin \mathcal{G}_{C_5}$, then we are done by symmetry of $G_1$ and $G_2$, so we suppose $G_2\in \mathcal{G}_{C_5}$.
As every vertex of $G_2-x$ has an even degree, $d_G(v)$ is even and thus $v$ satisfies the even condition.
By induction hypothesis, there is an odd 4-coloring $\varphi_1$ of $G_1$ such that $x$ satisfies the even condition.
Without loss of generality, we may assume that $\varphi_1(x)=1$ and $|\varphi_1^{-1}(2)\cap N_{G_1}(x)|$ is even.
By Proposition \ref{prop:c5family}, there is a proper 4-coloring $\varphi_2$ of $G_2$ such that every vertex $u\neq v$ satisfies the odd condition.
Without loss of generality, we may assume that $\varphi_2(x)=1$ and $|\varphi_2^{-1}(2)\cap N_{G_1}(x)|$ is odd.
Again  we define a coloring $\varphi$ of $G$ by $\varphi(u)=\varphi_i(u)$ for $i\in\{1,2\}$, and it is easy to check that $\varphi$ is a desired odd 4-coloring of $G$, a contradiction.
Therefore  $G$ is 2-connected.

Let $F$ be the outercycle of $G$.
If $G$ is a cycle, then it is shown by Caro et al.~\cite{Caro} that a cycle is odd 4-colorable unless its length is equal to 5, and each vertex has one color missing at its neighbors and hence satisfies the even condition.
Suppose that $F$ has at least one chord.
By Lemma \ref{lem:unavoidable}, $G$ contains an ear,  or an ear chain,  or an ear double chain that is good for $v$.

For a vertex $x\in V(G)$, we say $x$ satisfies the {\em parity condition}  if $x$ satisfies the odd condition, and in case $x=v$, then $x$ satisfies the even condition as well.

\begin{lemma}\label{lem:extension.ear}
    Suppose that $G$ has an ear $H$ with the root edge $u_1u_r$ such that $v\notin V(H)\setminus\{u_1\}$, and 
    $\varphi$ is a proper $4$-coloring of a subgraph of $G-(V(H) \setminus \{u_1,u_r\})$, in which $\{u_1, u_r\}$ and 
    all neighbors of $u_1$ in $G-(V(H)\setminus\{u_1, u_r\})$ are colored.
    Then $\varphi$ can be extended  to a proper 4-coloring of $H$ so that every vertex of $V(H)\setminus\{u_r\}$ satisfies the parity condition.
\end{lemma}

\begin{proof}
    Let $H$ be an ear with vertices $V(H)=\{u_1, u_2, \dots , u_r\}$ appearing in this order along $F$.
    %Let $G'=G-(V(H)\setminus \{u_1, u_r\})$ and, suppose that $\varphi(x)\in [4]$ is defined for every $x\in \{u_1, u_r\}\cup N_{G'}(u_1)$.
    We color vertices  $\{u_2, u_3, \dots , u_{r-1}\}$ in the ascending order of indices.
    For each $i\in\{2, 3, \dots , r-3\}$, let $\varphi(u_i) \in [4]\setminus\{\varphi(u_{i-1})\}$  so that $u_{i-1}$ satisfies the parity condition with respect to $\varphi$.
    For each $i\in\{r-2, r-1\}$, we choose a color in $[4]\setminus\{\varphi(u_{i-1}), \varphi(u_r)\}$ as $\varphi(u_i)$ so that $u_{i-1}$ satisfies the parity condition with $\varphi$.
    Then $\varphi$ is a desired coloring.
\end{proof}

\begin{lemma}\label{lem:extension.earchain}
    Suppose that $G$ has an ear chain $H$ with the root edge $v_1v_s$ such that $v\notin V(H)\setminus\{v_1\}$, 
    and $\varphi$ is a proper $4$-coloring of a subgraph of $G-(V(H) \setminus \{u_1,u_r\})$, in which $\{u_1, u_r\}$ and 
    all neighbors of $u_1$ in $G-(V(H)\setminus\{u_1, u_r\})$ are colored.
    Then  $\varphi$ can be extended to a proper 4-coloring of $H$ so that every vertex of $V(H)\setminus\{v_s\}$ satisfies the parity condition.
\end{lemma}

\begin{proof}
    Let $H$ be an ear chain of $G$ consists of the sequence of ears $H_1, H_2, \dots , H_{s-1}$, and let $v_iv_{i+1}$ be the root edge of $H_i$ for each $i\in [s-1]$.
    %Let $G'=G-(V(H)\setminus \{v_1, v_s\})$ and, suppose that $\varphi(x)\in [4]$ is defined for every $x\in \{v_1, v_s\}\cup N_{G'}(v_1)$.
    We first color $\{v_i\mid 2\leq i\leq s-1\}$ properly as $\varphi$.
    Using Lemma \ref{lem:extension.ear} to each ear $H_i$ in the ascending order of indices, we obtain a coloring $\varphi$ of $H$ such that every vertex of $V(H)\setminus\{v_s\}$ satisfies the parity condition.
\end{proof}

Now we show that all of unavoidable structures are reducible in odd 4-coloring.

\vspace{\baselineskip}
\noindent
\textbf{Case 1.} $G$ contains an ear $H$ good for $v$.

Let $H$ be an ear good for $v$ with the root edge $u_1u_r$, and let $G'=G-(V(H)\setminus\{u_1, u_r\})$.
We define a proper 4-coloring $\varphi'$ of $G'$ as follows:
If  $G'\simeq C_5$, then let $\varphi'$ be a proper 4-coloring of $G'$ such that every vertex $x \ne u_1$  satisfies the odd condition.
Otherwise, let $\varphi'$ be an odd 4-coloring of $G'$ such that the vertex $v$ satisfies the even condition.
Note that if $G'\simeq C_5$, then every vertex of $G'$ satisfies the even condition with respect to $\varphi'$. In particular, 
$v$ satisfies the even condition.
Let $\varphi(x)=\varphi'(x)$ for every $x\in V(G')$.
By Lemma \ref{lem:extension.ear}, we extend $\varphi$ to a proper 4-coloring of $G$ such that every vertex of $V(H)\setminus\{u_r\}$ satisfies the parity condition.
Since $d_G(u_r)=3$ and $u_r\neq v$, $u_r$ satisfies the parity condition as well, and thus $\varphi$ is a desired odd 4-coloring of $G$, a contradiction.\\

Before we go to the cases that $G$ contains either an ear chain or an ear double chain good for $v$, we consider the case that $G$ contains an ear with more than 5 vertices.

\vspace{\baselineskip}
\noindent
\textbf{Case 2.} $G$ contains an ear $H$ with the root edge $u_1u_r$ such that $v\notin V(H)\setminus\{u_1, u_r\}$ and $|V(H)|\geq 6$.

Let $V(H)=\{u_1, u_2, \dots , u_r\}$ appearing in this order along $F$, and let $G'=G-\{u_2, u_3, \dots , u_{r-1}\}$.
If $G'\simeq C_5$, then we have $d_G(u_r)=3$ and $G$ admits a desired odd coloring by Case 1. 
Hence we may assume that $G'\not\simeq C_5$.
By induction hypothesis, there is an odd 4-coloring $\varphi'$ of $G'$ such that $v$ satisfies the even condition.
Without loss of generality, we may assume that $\varphi'(u_r)=1$ and 
$|\varphi'^{-1}(2)\cap N_{G'}(u_r)|$ is odd.
Let $\varphi(x)=\varphi'(x)$ for every $x\in V(G')$, and let $\varphi(u_{r-2})=2$.
We choose colors for $\{u_2, u_3, \dots , u_{r-3}, u_{r-1}\}$ in the ascending order of indices.
For each $i\in\{2, 3, \dots , r-5\}$, we choose a color in $[4]\setminus\{\varphi(u_{i-1})\}$ as $\varphi(u_i)$ so that $u_{i-1}$ satisfies the parity condition.
For each $i\in\{r-4, r-3\}$, we choose a color in $[4]\setminus\{2, \varphi(u_{i-1})\}$ as $\varphi(u_i)$ so that $u_{i-1}$ satisfies the parity condition.
Finally, we choose a color in $\{3,4\}$ as $\varphi(u_{r-1})$ so that $u_{r-2}$ satisfies the parity condition.
By the choice of colors, $\varphi$ is a proper 4-coloring of $G$, and every vertex in $V(G)\setminus\{u_r\}$ satisfies the parity condition with respect to $\varphi$.
Furthermore, since $|\varphi^{-1}(2)\cap N_G(u_r)|=|\varphi'^{-1}(2)\cap N_{G'}(r_r)|$ is odd and $u_r\neq v$, $u_r$ satisfies the parity condition, so $\varphi$ is a desired odd 4-coloring of $G$, a contradiction.\\

In the following cases, we may assume that every ear of $G$ without $v$ in its internal vertices contains at most 5 vertices.

\vspace{\baselineskip}
\noindent
\textbf{Case 3.} $G$ contains a good ear chain $H$ with the root edge $v_1v_s$ such that $d_G(v_s)\in \{4, 5\}$.

Let $H_1, H_2, \dots , H_{s-1}$ be ears contained in $H$ where the root edge of $H_i$ is $v_iv_{i+1}$ for each $i\in [s-1]$.
If $d_G(v_s)=5$, then we derive a contradiction by applying the similar argument with in Case 1 to $H_{s-1}$.
Hence we may assume that $d_G(v_s)=4$.
Let $V(H_{s-1})=\{v_{s-1}, u_2, u_3, \dots u_{r-1}, v_s\}$ appearing in this order along $F$.
By the assumption after Case 2, we know that $r\leq 5$.

Let $G'=G-(V(H)\setminus\{v_1, v_s\})$ and we define a proper 4-coloring $\varphi'$ of $G'$ as follows:
If $G'\simeq C_5$, then let $\varphi'$ be a proper 4-coloring of $G'$ such that every vertex $x\neq v_1$ satisfies the odd condition.
Otherwise, let $\varphi'$ be an odd 4-coloring of $G'$ such that $v$ satisfies the even condition.
Without loss of generality, we may assume that $\varphi'(v_1)=1$ and $\varphi'(v_s)=2$.
Let $\varphi(x)=\varphi'(x)$ for every $x\in V(G')$.
We choose colors for $\{v_i\mid 2\leq i\leq s-1\}\cup V(H_{s-1})$ as follows.

\begin{itemize}
  \item[(a)] If $r=3$ and $s=3$, then let $\varphi(u_2)=3$ and $\varphi(v_2)=4$.
  \item[(b)] If $r=3$ and $s\geq 4$, then let $\varphi(v_{s-1})=1$, $\{\varphi(u_2), \varphi(v_{s-2})\}=\{3,4\}$ so that $v_s$ satisfies the odd condition with respect to $\varphi$, and choose colors for $\{v_i\mid 2\leq i\leq s-3\}$ properly.
  \item[(c)] If $r=4$ and $s=3$, then let $\varphi(u_3)=1$ and $\{\varphi(u_2), \varphi(v_2)\}=\{3,4\}$ so that $v_3$ satisfies the odd condition with respect to $\varphi$.
  \item[(d)] If $r=4$ and $s\geq 4$, then let $\varphi(v_{s-1})=1$, $(\varphi(u_2), \varphi(u_3), \varphi(v_{s-2}))\in\{(4,3,3), (3,4,4)\}$ so that $v_s$ satisfies the odd condition with respect to $\varphi$, and choose colors for $\{v_i\mid 2\leq i\leq s-3\}$ properly.
  \item[(e)] If $r=5$ and $s=3$, then let $\varphi(u_4)=\varphi(v_2)=3$, $\varphi(u_2)=4$, and $\varphi(u_3)=1$.
  \item[(f)] If $r=5$ and $s\geq 4$, then let $\varphi(u_4)=\varphi(v_{s-1})=3$, $\varphi(u_2)=1$, $\varphi(u_3)=\varphi(v_{s-2})=4$, and choose colors for $\{v_i\mid 2\leq i\leq s-3\}$ properly.
\end{itemize}

In any case, every vertex in $V(H_{s-1})\setminus\{v_{s-1}, v_s\}$ satisfies the odd condition with respect to $\varphi$.
Furthermore, since $d_G(v_{s-1})=4$ and 3 colors appear in the neighborhood of $v_{s-1}$, $v_{s-1}$ satisfies the odd condition with respect to $\varphi$ no matter what color appears at the neighbor of $v_{s-1}$ in $V(H_{s-2})\setminus\{v_{s-2}\}$.
In (b), (c) and (d), $v_s$ satisfies the odd condition with respect to $\varphi$ by the choice of colors.
In (a), $v_s$ satisfies the odd condition with respect to $\varphi$ since $d_G(v_s)=4$ and there are three distinct colors in the neighborhood of $v_s$.
In (e) and (f), $v_s$ satisfies the odd condition with respect to $\varphi$ since for every color $j\neq \varphi(v_s)$, $|\varphi^{-1}(j)\cap N_G(v_s)|$ and $|\varphi^{-1}(j)\cap N_{G'}(v_s)|$ have the same parity.
By applying a coloring in Lemma \ref{lem:extension.ear} to ears $H_1, H_2, \dots , H_{s-2}$ in the ascending order of indices, we extend $\varphi$ to a desired odd 4-coloring of $G$, a contradiction.

Note that the assumption that $d_G(v_s)=4$ is required only in (a), which implies the following statement.

\begin{quote}
  $(*)$ Suppose that $G$ contains an ear chain $H$ with the root edge $v_1v_s$ such that $v\notin V(H)\setminus\{v_1, v_s\}$.
  If either $H$ contains at least 3 ears, or $H$ contains an ear consists of at least 4 vertices, then $(G,v)$ is odd 4-colorable.
\end{quote}

\vspace{\baselineskip}
\noindent
\textbf{Case 4.} $G$ contains an ear double chain $H$ with the root edge $w_1w_t$ such that $v\notin V(H)\setminus\{w_1, w_t\}$.

Let $H_1, H_2, \dots , H_{t-1}$ be ear chains contained in $H$ where the root edge of $H_i$ is $w_iw_{i+1}$ for each $i\in [t-1]$.
Let $G'=G-(V(H)\setminus\{w_1, w_t\})$.
If $G'\simeq C_5$, then we have $d_G(w_1)=d_G(w_t)=5$, so one of 
$H_1$ and $H_{t-1}$ is a good ear chain for $v$ of $G$ and we are done by Case 3.
Hence we may assume that $G'\not\simeq C_5$.
By $(*)$, we may assume that $V(H_i)=\{w_i, u_i, v_i, u_i', w_{i+1}\}$ and $E(H_i)=\{w_iw_{i+1}, w_iv_i, v_iw_{i+1}, w_iu_i, u_iv_i, v_iu_i', u_i'w_{i+1}\}$ for every $i\in [t-1]$.

By induction hypothesis and Lemma \ref{lem:colorexchange}, there is an odd 4-coloring $\varphi'$ of $G'$ such that $v$ and at least one of $w_1$ and $w_t$ satisfy the even condition with respect to $\varphi'$.
By symmetry of $w_1$ and $w_t$, we may assume that $w_t$ satisfies the even condition.
Without loss of generality, we may assume that $\varphi'(w_1)=1$, $\varphi'(w_t)=2$, and $|\varphi'^{-1}(j)\cap N_{G'}(w_t)|$ and $|\varphi'^{-1}(k)\cap N_{G'}(w_t)|$ have the different parities for some $j, k\in \{1,3,4\}$.
Let $\varphi(x)=\varphi'(x)$ for every $x\in V(G')$.
Let $\varphi(u_{t-1}')=1$, $\varphi(v_{t-1})=3$, $\varphi(w_{t-1})=4$, and choose colors for $\{w_i\mid 2\leq i\leq t-2\}$ properly.
Then $w_t$ satisfies both the odd condition and the even condition with respect to $\varphi$ since $|\varphi^{-1}(j)\cap N_G(w_t)|=|\varphi'^{-1}(j)\cap N_{G'}(w_t)|+1$ and $|\varphi^{-1}(k)\cap N_G(w_t)|=|\varphi'^{-1}(k)\cap N_{G'}(w_t)|+1$ have the different parities. 
Furthermore, as $d_G(v_{t-1})=4$ and 3 colors appear in the neighborhood of $v_{t-1}$, $v_{t-1}$ satisfies the odd condition with respect to $\varphi$ no matter what color appears at $u_{t-1}$.
We can extend $\varphi$ to a desired odd 4-coloring of $G$ by applying a coloring in Lemma \ref{lem:extension.earchain} to $H_1, H_2, \dots , H_{t-2}$ in the ascending order of indices, and finally coloring $u_{t-1}$ by a color in $\{1, 2\}$ so that $w_{t-1}$ satisfies the parity condition with respect to $\varphi$, a contradiction.
This completes the proof of Case 4 and the proof of Theorem \ref{thm:main2}.

\section*{Acknowledgement}
Masaki Kashima is supported by Keio University SPRING scholarship Grant number JPMJSP2123 and  JST ERATO Grant Number JPMJER2301. Xuding Zhu is supported by Grant numbers:  NSFC 12371359, U20A2068.


\begin{thebibliography}{99}
  \bibitem{Anderson}
  J. Anderson, H. Chau, E.-K. Cho, N. Crawford, S. G. Hartke, E. Heath, O. Henderschedt, H. Kwon, and Z. Zhang, The forb-flex method for odd coloring and proper conflict-free coloring of planar graphs,
  arXiv:2401.14590v1.
  %\bibitem{Bondy}
  %J. A. Bondy and U. S. R. Murty, Graph Theory, Graduate Texts in Mathematics, vol. 244, Springer, New York, 2008.
  \bibitem{Caro}
  Y. Caro, M. Petru\v{s}evski, and R. \v{S}krekovski, Remarks on odd colorings of graphs,
  \textit{Discrete Appl. Math.} 321 (2022), 392-401.
  \bibitem{Caro2}
  Y. Caro, M. Petru\v{s}evski, and R. \v{S}krekovski, Remarks on proper conflict-free colorings of graphs,
  \textit{Discrete Math.} 346 (2023), no. 2, 113221.
 % \bibitem{Chen}
 % P. Chen and X. Zhang, Improper odd coloring of IC-planar graphs, \textit{Discrete Appl. Math.} 357 (2024), 74-80.
  \bibitem{Cho}
  E.-K. Cho, I. Choi, H. Kwon, and B. Park, Odd coloring of sparse graphs and planar graphs,
  \textit{Discrete Math.} 346 (2023), no. 5, 113305.
  \bibitem{Kashima}
  M. Kashima, S. Maezawa, K. Osako, K. Ozeki, and S. Tsuchiya, An odd 4-coloring of a maximal outerplanar graph, manuscript.
  \bibitem{Liu}
  R. Liu, W. Wang, and G. Yu, 1-planar graphs are odd 13-colorable, \textit{Discrete Math.} 346 (2023), no. 8, 113423.
%  \bibitem{Miao}
%  Z. Miao, L. Sun, Z. Tu, and X. Yu, On odd colorings of planar graphs,
%  \textit{Discrete Math.} 347 (2024), 113706.
  \bibitem{Petr}
  J. Petr and J. Portier, the odd chromatic number of a planar graph is at most 8, \textit{Graphs Combin.} 39 (2023), 28.
  \bibitem{Petrusevski}
  M. Petru\v{s}evski and R. \v{S}krekovski, Colorings with neighborhood parity condition,
  \textit{Discrete Appl. Math.} 321 (2022), 385-391.
  \bibitem{Wang}
  T. Wang and X. Yang, On odd colorings of sparse graphs,
  \textit{Discrete Appl. Math.} 345 (2024), 156-169.
\end{thebibliography}
\end{document}